\documentclass{elsarticle}
\usepackage{geometry}
\usepackage{bm}
\usepackage{amsmath}
\usepackage{amsthm}
\usepackage{amssymb}
\usepackage{tikz}
\usepackage{mathtools}

\usepackage{float}
\usepackage{tabularx}

\usepackage{caption,subcaption}
\usepackage{color}

\newtheorem{theorem}{\textbf{ Theorem} }
\newtheorem{lemma}{\textbf{     Lemma } }
\newtheorem{definition}{\textbf{Definition} }

\newtheorem{remark}{\textbf{Remark} }

\begin{document}

\title{The Z-eigenpairs  of  orthogonally    diagonalizable  symmetric  tensors
\tnoteref{mytitlenote}}
\tnotetext[mytitlenote]{This research did not receive any specific grant from funding agencies in the public, commercial, or
not-for-profit sectors.}

	\author[add1,add2,add3]{Lei  Wang}

\author[add1,add2,add3]{Xiurui Geng\corref{coraut}}
\cortext[coraut]{Corresponding author}
\ead{gengxr@sina.com.cn}

\address[add1]{Aerospace Information Research Institute, Chinese Academy of Sciences, Beijing 100094, China}
\address[add2]{University of the Chinese Academy of Sciences, Beijing 100049, China} 
\address[add3]{Key Laboratory of  Technology in Geo-Spatial Information Processing and Application System, Chinese Academy of Science, Beijing 100190, China}
	
\begin{abstract} 
	In  this paper,   we  focus   on  a special  class    of  symmetric  tensors,  which can be    orthogonally  diagonalizable,  and     investigate     their   Z-eigenpairs  problem.   We show   that  the  eigenpairs  can be  uniformly   expressed  using  several   basic  eigenpairs,  and  the number of  all  the   eigenpairs   is  
	uniquely  determined  by   the  order  and  rank  of  the   symmetric  tensor.  In  addition, we    exploit  the  local  optimality  of each  eigenpair  by  checking     the second-order  necessary  condition.  
\end{abstract}

\maketitle

\section{Introduction }

As  the high-order  generalization of  matrix,  tensor   analysis and  applications have  been  given   more   attentions in  recent  years \cite{kolda,2007Numerical}.   
Many  concepts  have   been  naturally   extended  from    matrices  to  tensors,  such  as  the  inner  product of   the tensor\cite{kolda},
  tensor norm\cite{Tnorm},
 tensor  rank\cite{tensorrank,kolda,tensorrank_qi},
high-order Sigular value decomposition (HOSVD)\cite{hosvd}, etc.   
However,  it  has  also  been   analyzed  that  most of tensor problems are NP hard \cite{NP-hard},    one of which   is  to  obtain  all   the  eigenpairs  of symmetric  tensors\cite{qi,lim}.
Different  from  the matrix  case,    there are several  definitions for eigenpairs  of symmetric tensors, such as   D-eigenpairs\cite{Deigen}, H-eigenpairs\cite{qi}, Z-eigenpairs\cite{qi},   etc.  
 In this paper,  we mainly  focus  on  one of them --- E-eigenpairs  and when  the  corresponding  eigenvector  is  real, it  is  also  called Z-eigenpairs.

There  are  many  algorithms  and   numerous  applications  that  have  been  investigated  concerning  this  subject.
  see for example  \cite{SHOPM,ASHOPM,PSA, NPSA,NCM,OTD, Cuicf,hm,Zglobal}.   
In this  paper,  we mainly  investigate the  Z-eigenpairs  problem    of      a  special  class  of  symmetric  tensors, which  is  orthogonally   diagonalizable.
We   first  prove  that   the   eigenpairs of  such  a  type of  tensors   can  be  enumerated  in  a linear-combination   way     using     several  basic  eigenvectors,  and  the  number  of eigenpairs  can   be  uniquely   determined  by  the  order  and  rank  of   the  symmetric  tensor.
In  addition,     the   local  optimality  of   each  eigenpair   is   also   analyzed  by checking  the  second-order  necessary   condition.

\section{ Preliminaries}
We  start  by   defining  some  notations.
Let    $ \mathbb  C $  and $ \mathbb  R $    be  the  complex and  real  field.
High-order   tensors are denoted
boldface Euler script letters, e.g.,
$\mathcal  A $.
An
$d$th-order tensor is denoted 
$\mathcal A \in \mathbb {R}^{I_1 \times I_2  \times \dots \times I_{d} }$,
where 
$d$
is the order (or  the  way, or  the mode) of
$\mathcal A $, and 
$ I_j $ ($  j \in \{ 1,2,\dots,d \}$)  is  the  dimension  of  
$j$th-mode.
 The element of $\mathcal A$,    which  is  indexed  by integer tuples $(i_1,i_2,\dots,i_d) $, is denoted 
($a_{i_1,i_2,\dots,i_d})_
{1 \le i_1 \le I_1, 
	1 \le i_2 \le I_2,
	\dots, 
	1 \le i_d \le I_d}$. 
When  $d=2$,     a  tensor   reduces  to  a  matrix.    When  $d=1$,      it  is  a  vector. 
 The $j$-mode product of a tensor
$\mathcal A \in \mathbb {R}^{I_1 \times I_2 \times \dots \times I_d } $
with a matrix
$\mathbf B \in \mathbb {R}^{J \times I_j}$
is   denoted by
$\mathcal A 
\times_{j} 
\mathbf B 
\in \mathbb {R}^{I_1 \times \dots \times I_{n-1} \times J \times I_{n+1}\times \dots \times I_d } $,
whose element is
\begin{align}
(\mathcal A \times_{j} \mathbf B )_{i_1\dots i_{n-1}ji_{n+1}\dots i_d}=\sum_{i_{j}=1} ^{I_{j}} a_{i_1,i_2,\dots,i_d} b_{ji_{n}} .
\end{align}
For   different   modes   in  series  of   multiplications,   the  order   is   cummutable\cite{kolda},  i.e.,  
\begin{equation} \label{tensormodemn}
\mathcal A 
\times_{j} 
\mathbf B 
\times_{k} 
\mathbf C 
= 
\mathcal A 
\times_{k} 
\mathbf C 
\times_{j}  
\mathbf B
\quad 
(j \neq k). 
\end{equation}
If   the  modes  are  the  same,  it  holds  that
\begin{equation} \label{tensormodenn}
\mathcal A 
\times_{j} 
\mathbf B 
\times_{j} 
\mathbf C 
 = 
 \mathcal A \times_{j} ( \mathbf C\mathbf B).
\end{equation}
Given  a    $d$th-order  tensor  
$  \mathcal A   $    and   a  series  of     matrices  
$\mathbf B^{(i) }  (i=1,2\dots, d)$,  
it is   simply  denoted  
\begin{align}
 \mathcal A  
 \times_{1 }    \mathbf B^{(1) }    
  \times_{2 }    \mathbf B^{(2) }  
  \times_{3 }     
  \dots
\times_{d }
\mathbf B^{(d) }
=
[\mathcal A;   
\mathbf B^{(1) },     
\mathbf B^{(2) }, 
\dots,
\mathbf B^{(d) }],
\end{align}   which  is   a     notation   introduced by  Kolda  in  \cite{kolda}.

A   tensor is called    symmetric (or supersymmetric) if its elements remain invariant under any permutation of the indices $(i_1,i_2,\dots,i_d) $\cite{kolda}.  A   symmetric  tensor of order
$ m $
and dimension
$ n$
is denoted 
$\mathcal S \in \mathbb R^{n \times n  \times \dots \times n } $,
whose element is
 \begin{equation} 
\mathcal S = (
s_{i_1,i_2,\dots,i_m})_
{ 1 \le i_1 \le n, 
	1 \le i_2 \le n,
	\dots, 
	1 \le i_m \le n},
 i_j \in \{{1,2,\dots,n}\}, j=1,2,\dots,m.
 \end{equation} 
Let    $  T^{m}(\mathbb R^{n}) $  denote  the  space  of  all  such  real  symmetric    tensors.
Given   a   vector  $  \mathbf u \in  \mathbb  C^{n} $  and   a    symmetric  tensor 
$\mathcal S   \in   T^{m}(\mathbb R^{n})$,  a  series  of  multiplication  along  different   modes   can   be   simply  denoted   as  follows:
\begin{equation}\label{sum_simple}
 \mathcal S  
\times_{1 }    \mathbf u^{\mathrm T} 
\times_{2 }    \mathbf u^{\mathrm T} 
\times_{3 }     
\dots
\times_{m }
 \mathbf u^{\mathrm T}
 =
\mathcal S \mathbf u^{m} =
\sum\limits_{i_1,i_2,\dots,i_m=1}^{n} 
s_{i_1,i_2,\dots,i_m}  u_{i_1} \dots   u_{i_{m}}.
\end{equation}
And  in  a  similar  way,   $   \mathcal S \mathbf u^{m-1}   $   denotes   an    $n$-dimensional
column   
vector,  whose  $j$th  element   is    
\begin{equation}\label{sum1}
(\mathcal S \mathbf u^{m-1})_{j} =
\sum\limits_{i_2,\dots,i_m =1}^{n} 
s_{j,i_2,\dots,i_m}  u_{i_2} \dots   u_{i_{m}}.
\end{equation}

Furthermore,   $   \mathcal S \mathbf u^{m-2}   $   is  an    $n  \times  n $  matrix,    whose  $(i, j )$th  element   is  
\begin{equation}\label{sum2}
(\mathcal S \mathbf u^{m-2})_{i, j} =
\sum\limits_{i_3,\dots,i_m =1}^{n} 
s_{i,j,i_3,\dots,i_m}  u_{i_3} \dots   u_{i_{m}}.
\end{equation}

 Given   a  matrix  
$ \mathbf A  \in \mathbb {R}^{n \times k} $   and    assuming  that  
$ \mathbf A $  is  full   column  rank ($  k  \le  n$), 
the orthogonal  complement   projection matrix of  $ \mathbf A $ 
is  denoted  
$ \mathbf {P}^{\bot}_{  \mathbf A } 
= 
\mathbf I_{n}-\mathbf A (\mathbf A^{\mathrm T} \mathbf A)^{-1} \mathbf A^{\mathrm T}$, 
where  we  use    $ \mathbf I_{n} $   to  denote    an $ n  \times  n $   identity matrix.

Given   $d$  vectors 
$ \mathbf a^{(i) } \in \mathbb {R}^{I_i \times 1}$ 
($i=1,2, \dots, d$),
their   outer  produce   
$ \mathbf a^{(1) }
\circ
\mathbf a^{(2) }
\circ  \dots
\circ
\mathbf a^{(d) } 
$  
is   a    $ d$th-order  tensor,
 denoted  $ \mathcal X  \in \mathbb {R}^{I_1 \times I_2  \times \dots \times I_d } $ ,
 whose   element is    the  product  of    the  corresponding  vector   element:
\begin{align} 
x_{i_1,i_2,\dots,i_d}
= 
\mathbf a^{(1) }_{i_1}
\mathbf a^{(2) }_{i_2}
  \dots
\mathbf a^{(d) }_{i_d} 
,
{1 \le i_1 \le I_1, 
	1 \le i_2 \le I_2,
	\dots, 
	1 \le i_d \le I_d}.
\end{align} 
The  tensor  $ \mathcal X$  is  called  to  be   rank-one    if  it  can  rewritten  as  the  outer  product  of   $d$  vectors.
When 
$  \mathbf a^{(1) }
=  
\mathbf a^{(2) }
=   \dots
= 
\mathbf a^{(d) }
=\mathbf a $ $(I_1 =  I_2  = \dots = I_d)$,  we use  the  notation 
$ \mathcal X =  \mathbf a^{\circ d}$  for  simplicity,   where 
$   \mathcal    X $ is  a  symmetric  tensor  of  order  $d$  and  dimension  $ I$.  
Let  $ \mathbb  A $   be  the  set  of   $ k$  $(1 \le  k  \le  r)$  integers  randomly   selected  from   the  set    of   $r$    integers  $ \{  1,2,\dots, r \} $.
$\vert    \mathbb  A  \vert  =k     $  denotes  the number  of  the elements  in  $ \mathbb  A $.

In  this paper,   the  following  optimization  model is   considered:
\begin{equation}\label{opti_ori}
\begin{cases}
\max\limits_{\mathbf u} \quad \mathcal S \mathbf u^{m}   \\
\rm s.t. \quad \mathbf u^{\mathrm {T}}\mathbf u=1
\end{cases}.
\end{equation}
The Lagrangian function 
of (\ref{opti_ori})  is  defined as:
\begin{equation}\label{Lagrangian_function}
\rm L(\mathbf u, \lambda)=
\frac {1}  {\it m}
\mathcal S \mathbf u^{\it m}-
\frac { \lambda} {2} (\mathbf u^{\mathrm {T}}\mathbf u-1).
\end{equation}
When the gradient of
$ \rm L(\mathbf u, \lambda) $ to
$ \mathbf  u $ is  $ \mathbf 0$,    the eigenpair of a  symmetric  tensor can  be  deduced,  which  was  independently    defined  by   Lim  and  Qi  in  2005:

\begin{definition}\cite{qi,lim}
	Given a  tensor $\mathcal S   \in    T^{m}(\mathbb R^{n}) $,
	a pair
	$(\lambda ,\mathbf u )$
	is an Z-eigenpair  of  
	$\mathcal S  $ 
	if
	\begin{equation}\label{definition}
	\mathcal S \mathbf u^{m-1}=\lambda \mathbf u,
	\end{equation}
	where
	$ \lambda  \in  \mathbb C $
	is  the  eigenvalue and
	$ \mathbf u  \in   \mathbb R^{n \times  1} $
	is the  corresponding   eigenvector   satisifying 
	$\mathbf u^{\mathrm {T}}\mathbf u=1 $.
\end{definition}

Assuming    that     $ (\lambda,  \mathbf  u)  $ is     an  Z-eigenpair  of   $\mathcal S   \in    T^{m}(\mathbb R^{n}) $,  and  it  is   easily   checked   that 
so  is   
$ (	\lambda' =  t^{m-2}\lambda,
\mathbf  u' =  t \mathbf  u)
$  
for  
$  	t \in  \mathbb  C \backslash \{0 \}$.  
This   means   that   the  solution   of  
(\ref{definition}) 
consists
of different equivalence classes.     Such  an   equivalence  class      is   denoted  as   follows:

\begin{definition}\label{euqclass}
	Let  $ (\lambda,  \mathbf  u)  $ be  an  Z-eigenpair  of   $\mathcal S   \in    T^{m}(\mathbb R^{n}) $,     
	the   equivalence  class   of     $ (\lambda,  \mathbf  u)  $   is  denoted   
	\begin{equation}\label{equclass}
	[  
	(\lambda,  \mathbf  u) ]: =
	\{
	(\lambda',  \mathbf  u')  \vert 
	\lambda' =  t^{m-2}\lambda  ,
	\mathbf  u' =  t \mathbf  u ,
	t \in  \mathbb  C \backslash \{0 \}
	\}.
	\end{equation}
\end{definition}

Assume that
(\ref{opti_ori})  or   (\ref{definition}) 
has 
finite solutions. 
Then, 
the  following  theorem  provides  a  theoretical  upper  bound  for  the number of  	Z-eigenpairs:
\begin{theorem}\cite{upperbound}
	If a tensor $\mathcal S   \in    T^{m}(\mathbb R^{n})  $  has finitely many equivalence classes of z- Z-eigenpairs over 
	$ \mathbb  C$,   then their number,
	counted with multiplicity,     is   bounded  by  
	\begin{equation}\label{at_most}
	M(m,n)=
	\frac {(m-1)^{n}-1}  {m-2}.
	\end{equation}
\end{theorem}

Cartwright and Sturmfels  firstly   proves  the  above    theorem\cite{upperbound}.  
In  the  literature  of  \cite{hm},  the  authors  also  provided  another  version  of  the  proof,  which  considers   various   
types   of  tensor   eigenvalues   and    provide  an  unified   results. 

The  first-order    gradient  derivation     (\ref{definition})   can be  used  to    obtain        all   stationary  points  of  (\ref{opti_ori}).   
Whie   the  second-order  derivation  information  plays  an  important  role  in  identifying    whether  a  stationary  point  is     locally    optimal    given    an  optimization   model.      The    second-order derivation  
of
$ \rm L(\mathbf u, \lambda) $ 
to    $ \mathbf u $, i.e.,  the Hessian matrix of
(\ref{Lagrangian_function}),  is   denoted
\begin{equation}\label{hessian_matrix}
\mathbf H(\mathbf u) = (m-1)\mathcal S \mathbf u^{m-2} - \lambda \mathbf I_{n} ,
\end{equation}
where
$ \mathbf I_{n} $
is an $ n \times n $  identity matrix.
 The  following  theorem is  well  established     for   the   constrained  optimization  problem   to  identify  the  locally  optimal   solutions
(Page 332 in   \cite{Numerical}): 

\begin{theorem}[\textbf{Second-order necessary condition}]\label{second_order_necessary}\cite{Numerical}
	Suppose that
	for any  vector $ \mathbf w \in \mathbb V $,
	if 
		\begin{equation}\label{second_order}
		\mathbf w^{\mathrm T}
	\mathbf H (\mathbf u) 
	\mathbf w  
	\le 0   
		\end{equation}
	holds, then
	$\mathbf u $
	is a local maximum solution of (\ref{opti_ori}).
	And for  (\ref{opti_ori}), the set
	$\mathbb V $
	is defined as
	\begin{equation}\notag
	\mathbb V=\{
	\mathbf w \in \mathbb R^{n}
	\vert   (\triangledown  g)^{\mathrm T} \mathbf w   =0
	\}=
        	\rm Null [(\triangledown  g)^{\mathrm T} ],
	\end{equation}
	where   $ \rm Null( \mathbf A) $ 
	denotes the null space  of $\mathbf A$  and $  \triangledown  g =\mathbf u $  denotes the gradient of the constraint: $ g (\mathbf u) = 
	\mathbf u^{\mathrm T}
	\mathbf u 
	-1
	=0 $.  
	
	If a stronger condition, i.e.,
	$ \mathbf w^{\mathrm T}
	\mathbf H (\mathbf u) 
	\mathbf w  <   0 $,   
	is satisfied, then,  
	$\mathbf u$
	is a   strict     local maximum  solution of (\ref{opti_ori}).
\end{theorem}

In  addition, instead of  directly utilizing  Theorem 
\ref{second_order_necessary},  
it  is  more  preferred to  identify the  local  extremum  by  checking  the  positive or negative  definiteness of  the  projected  Hessian  matrix (denoted 
$\mathbf P 
\in \mathbb R^{(n-1) \times  (n-1)} 
$).  It  is  calculated  by  
$  \mathbf P = 
\mathbf Q_{2}^{\mathrm T}
\mathbf H(\mathbf u) 
\mathbf Q_{2}$,  
where 
$  \mathbf Q_{2} $ is  obtained by
QR  factorization of   $     \triangledown  g $:
\begin{equation}\label{QR_factor}
\begin{split}
\triangledown  g
&=\mathbf Q
\begin{bmatrix}
\mathbf R   \\
\mathbf 0
\end{bmatrix} =
\begin{bmatrix}
\mathbf Q_{1}  &   \mathbf Q_{2}
\end{bmatrix}
\begin{bmatrix}
\mathbf R  \\
\mathbf 0
\end{bmatrix}
=\mathbf Q_{1}
\mathbf R
\end{split},
\end{equation}
where 
$  \mathbf Q $ 
is   an   
$ n \times n $  
orthogonal  matrix, and $ \mathbf R$ is a square upper triangular matrix. 
  In this  case, 
  $ \mathbf R$ is  reduced  to  a  scalar  since  
  $ \triangledown  g $  is  a  column  vector  and  
  $\mathbf 0$ is  an  $ (n-1) \times   1 $  vector  with  all  elements  equal to 0.  
$  \mathbf Q_{1} $  and   
$  \mathbf Q_{2} $
are   
$ n \times 1$,   $ n \times (n-1) $ matrix, respectively. See more  details for  this  part in  Page 337 of \cite{Numerical}. 

The  concept  of  the  projected  Hessian  matrix   has   been    widely  researched   for  tensor  eigenvalues  problem. 
In  \cite{SHOPM},  T.G.Kolda  defines  that    an eigenvector
$ \mathbf u$   was termed positive-stable if
the   corresponding   projected  Hessian  matrix 
 $ \mathbf P$ is positive-definite,   and negative-stable if $ \mathbf P$ is negative-definite.
The  authors   in  \cite{NCM}   proposed  an  algorithm  termed    Orthogonal Newton correction method (ONCM),  where  the  projected  Hessian  matrix  is    calculated      in each   iteration  update  step. 

Using  these   preliminaries,   we   are  interested  in  analyzing  the Z-eigenpair  problem  of   a    special  class  of  symmetric  tensors, which can be  orthogonally diagonalizable. And  the  difinition is as   follows:
\begin{definition}\label{odtdef}\cite{sr1}
	Given a  symmetric  tensor $\mathcal S  \in    T^{m}(\mathbb R^{n}) $,   if  there  exists  a   matrix 
	$\mathbf U = [\mathbf u_{1}, \mathbf u_{2}, \dots, \mathbf u_{r}] \in   \mathbb  R^{n   \times  r}$, composed  of $r$ ($r \le n$)   orthonormal  vectors $ \mathbf u_{i}$ $(i=1,2,\dots, r )$,   
	and  a   diagonal  tensor 
	$   \mathcal D  \in    T^{m}(\mathbb R^{r})     $  
	with     $ \mathcal  D_{i,i,\dots,i}=   \lambda_{i} >  0  $, 
	and     it  holds  
	\begin{align}\label{odt_sum}
	\mathcal  S=  
	 [\mathcal D;   
	\mathbf U ,     
\mathbf U , 
\dots,
\mathbf U]
	=
	\sum_{i=1}^{r}   \lambda_{i} \mathbf  u_{i}^{\circ m}, 
	\end{align} 
	then  $\mathcal S    $ is    called 
	orthogonally  diagonalizable  symmetric  tensor,  which  is   the  summaion  of    $r$     rank-one   tensors   $\mathbf  u_{i}^{\circ m}  \in    T^{m}(\mathbb R^{n}) $ $(i=1,2,\dots, r )$.
\end{definition}

\begin{remark}	
	  The   rank  of   a  tensor  $ \mathcal X$  is defined as the smallest number of rank-one
	tensors   that   generate     $ \mathcal X$  as   their  sum. 
	For  the  symmetric   tensor  $\mathcal S    $ which   can  be    orthogonally  diagonalizable,     it   holds  that 
	$ rank(\mathcal S)  =  rank(\mathbf U)  =r $. 
\end{remark}

Due  to  its  special   structural  property,  the  following  lemma  holds:
\begin{lemma}\label{lemman}
	Let   
	$\mathcal S$
	and 
	$\mathcal D$  $\in    T^{m}(\mathbb R^{r}) $,  
	Let   
	$\mathbf U   \in    \mathbb  R^{n   \times  r} $  be   a      matrix  composed  of $r$   orthonormal  vectors $ \mathbf u_{i}$ $(i=1,2,\dots, r)$,  and   (\ref{odt_sum})
	holds.   Then,  
	$ (\lambda_{i},   \mathbf  u_{i} )  (i=1,2,\dots, r )$    
	is  an  Z-eigenpair  of     the    orthogonally  diagonalizable symmetric   tensor $\mathcal S   $.
\end{lemma}

\begin{proof}
	Since
	$    \mathbf  u_{i}^{\mathrm T}  \mathbf  u_{j} = 0   $    for $ \forall $   $   i \neq j $,  thus      
	\begin{equation}\label{odt_eig}
	\mathcal  S   \mathbf  u_{i}^{m-1}
	= 
	(  \sum_{i=1}^{r}   \lambda_{i} \mathbf  u_{i}^{\circ m}   )  \mathbf  u_{i}^{m-1}
	=      (\lambda_{i} \mathbf  u_{i}^{\circ m})    \mathbf  u_{i}^{m-1}
	=   \lambda_{i}  (  \mathbf  u_{i}^{\mathrm T}\mathbf  u_{i})^{m-1}          \mathbf  u_{i}
	=   \lambda_{i}        \mathbf  u_{i}  ,
	\end{equation}
	which  implies that  the conclusion  holds.
\end{proof}

\section{Main  results }

Based  on   Lemma  \ref{lemman}, we  first  show  that   the  Z-eigenpair  of  $  \mathcal S$  can  be   uniformly  expressed  as  follows:

\begin{lemma}\label{result_combination}
	Let   
	$\mathcal S$
	and 
	$\mathcal D$  $\in    T^{m}(\mathbb R^{r}) $ 	with     $ \mathcal  D_{i,i,\dots,i}=   \lambda_{i} >  0  $. 
	Let   
	$\mathbf U   \in    \mathbb  R^{n   \times  r} $  be   a   matrix  composed  of $r$   orthonormal  vectors $ \mathbf u_{i}$ $(i=1,2,\dots, r)$,  and   (\ref{odt_sum})
	holds. 
	For 
	$  i \in  \mathbb A$,
	we   assume  that  
	$\tilde { \mathbf  u} 
	=  \sum_{i \in  \mathbb A}  
	\tilde { c}_{i}   			\mathbf  u_{i} $  is   a  linear combination  of  	$\mathbf  u_{i}  $,   
	where   the  coefficients  
	$   \tilde { c}_{i}
	= 
	\sqrt[\uproot{3} {m-2}]
	{
		\frac  
		{1} 
		{          \lambda_{i}     }    
	}      $.   
	Denote  
	$  l=
	\Vert \tilde { \mathbf u }\Vert 
	= \sqrt  
	{ \sum_{i \in  \mathbb A     } 
		\tilde { c}_{i}^{2}  } $,   
	then  
	$   ( \lambda= 
	\frac  { 1}   {    l^{m-2}          },  
	\mathbf  u= 
	\frac  
	{ \tilde { \mathbf  u}    } 
	{        l    }  
	)    $ 
	is  an   eigenpair  of 
	$     \mathcal  S   $.
\end{lemma}

\begin{proof}
	We  start   by  computing 
	\begin{equation}\label{sum1odt}
	\mathcal S     
	\tilde { \mathbf  u}^{m-1}
	=
	\mathcal S  
	(  \sum_{    i \in  \mathbb A   } 
	\tilde { c}_{i} 
	\mathbf  u_{i}  )^{m-1}
	=
	\sum\limits_{    i \in  \mathbb A}  
	\tilde { c}_{i}^{m-1}  
	\mathcal S 
	\mathbf  u_{i} ^{m-1}  
	=
	\sum\limits_{   i \in  \mathbb A   }  
	\tilde { c}_{i}^{m-1}    \lambda_{i}  \mathbf  u_{i}     .   
	\end{equation}
	When  
		\begin{equation}\label{cim1}
	 \tilde { c}_{i}^{m-1}    \lambda_{i} =  \tilde { c}_{i}  , 
	 \end{equation}
is   satisfied,  i.e., 	$   \tilde { c}_{i}
	= 
	\sqrt[\uproot{3} {m-2}]
	{
		\frac  
		{1} 
		{          \lambda_{i}     }    
	}      
$,    
	it  holds     that 
	$\mathcal S   
	\tilde {\mathbf  u}^{m-1} =   \tilde {\mathbf  u}$.  
	Normalize  it  into a unit  length, 
	it   can  be  derived   
	\begin{equation}\label{normalization}
	\mathcal S  (  
	\frac  { \tilde { \mathbf  u}     }   {l}
	)^{m-1}=  
	\frac  { 1}   {l^{m-2}}  
	\frac  
	{ \tilde { \mathbf  u}   }   
	{l}    , 
	\end{equation}
	which   implies  that   $   ( \lambda=  \frac  { 1}   {l^{m-2}},  \mathbf u= \frac  {\tilde { \mathbf  u} }  {l}  )    $
	is  an  Z-eigenpair   of   $\mathcal S $. 
\end{proof} 


An  equivalent   matrix  expression  is :   
\begin{equation}\label{uUc}
  \mathbf u  =    \mathbf U   \mathbf c, 
  \end{equation} 
   where   
 $  \mathbf c  = 
\begin{bmatrix}
\mathbf c_{\mathbb  A} 
\\  \mathbf 0_{r-k}
\end{bmatrix} \in  
\mathbb R^{r \times  1} $, 
$ \mathbf c_{\mathbb  A}  \in  
\mathbb R^{k \times  1}$,  
$ \mathbf 0_{r-k}  \in  
\mathbb R^{(r-k) \times  1}$, 
and  
\begin{equation}\label{ci}
c_{i}  = 
\frac  
{   \tilde c_{i}      }
{  l  }  
=
\frac  
{    \sqrt[\uproot{3} {m-2}] {
		\frac  {1} {\lambda_{i} }    }      }
{  l  } ,  \quad    
i  \in   \mathbb  A.  
\end{equation} 
%

Lemma  \ref{result_combination}  shows  that  the  eigenpairs  of  
$\mathcal S $  can  be  uniformly  expressed   using   $  (\lambda_{i},   \mathbf  u_{i} )$  $(i=1,2,\dots, r )$, which  is  termed  the  basic  eigenpairs for  the  symmetric  tensor $\mathcal S $.  And  all  the  eigenvectors   is  a    linear  combination  of  the   basic  eigenvectors where  the  $k$ coefficients  are   given  in  (\ref{ci})   while the  left  
$r-k$ coefficients  are  equal  to  0.
Here, for  convenience  in the  later  analysis,    we 
do not  arrange  the  $ r$  eigenvectors  from  
$\mathbf  u_{1}$  to  $\mathbf  u_{r}$,  but      always   preferentially    arrange  the  $k$  participated   eigenvector.  
It is  easily  checked  that   
\begin{equation}\label{lmdici}
\lambda_{i}  c_{i}^{m-2}  =
\lambda_{i} 
(\frac  
{    \sqrt[\uproot{3} {m-2}] {
		\frac  {1} {\lambda_{i} }    }      }
{  l  }  
)^{m-2}
=
\lambda >0,
\quad 
i  \in   \mathbb  A. 
\end{equation}  

When 
$\vert    \mathbb  A  \vert  =k=1     $,  
it  falls into   
$  (\lambda_{i},   \mathbf  u_{i} )  (i=1,2,\dots, r)$.
Naturally,  the  following     question      to  be  answered  is    how  much   is  the  number  of  all  the   eigenpairs,  and  we   have  the  following  lemma:
\begin{lemma}
	Given a  symmetric  tensor $\mathcal S  \in    T^{m}(\mathbb R^{n}) $,   if  there  exists  a   matrix 
$\mathbf U = [\mathbf u_{1}, \mathbf u_{2}, \dots, \mathbf u_{r}] \in   \mathbb  R^{n   \times  r}$, composed  of $r$ ($r \le n$)   orthonormal  vectors $ \mathbf u_{i}$ $(i=1,2,\dots, r )$,   
and  a   diagonal  tensor 
$   \mathcal D  \in    T^{m}(\mathbb R^{r})     $  
with     $ \mathcal  D_{i,i,\dots,i}=   \lambda_{i} >  0  $, 
and    (\ref{odt_sum})  holds,  then,   the  number  of  all  the  eigenpairs  is   given  by  $\frac{(m-1)^{r}-1}{m-2} $. 
\end{lemma}

\begin{proof}
	It  can be  seen   from   (\ref{ci})  that  $     c_{i}  $    can   be  chosen  to  be  one  of  the $ (m-2) $    roots  of   $ \frac  {1} {\lambda_{i}} $ ($ i=1,2,  \dots,  r $). 
	Since $  k  $    is    randomly chosen  form  $ \{1,2,  \dots,  r\} $,       the total   number   of  the  combinations  is 
	\begin{equation}\label{totalnumber}
	\begin{split}
	\sum\limits_{k=1}^{r}
	\left(
	\begin{array}{c}
	r   \\
	k
	\end{array}  \right) 
	(m-2)^{k}
	&=  
	(m-2+1)^{r} - 
	\left(
	\begin{array}{c}
	r   \\
	0
	\end{array}  \right)   
	= (m-1)^{r}-1
	\end{split} .
	\end{equation}

	Based  on    Definition  \ref{equclass},  there
	are $ m-2$ distinct members of every equivalence class, and it  is enough for   each   equivalence  class  to   find  one  of  them [Theorem  5 in \cite{fasthm}].  So   the  number  of  all  the  eigenpairs is   given  by  $\frac{(m-1)^{r}-1}{m-2} $,  which  is  determined  by  the  order  and  rank  of  the   symmetirc  tensor. 
\end{proof}

\begin{remark}
	When  	$ rank(\mathcal S)  =  rank(\mathbf U)  =r =n $,     it is  easily  checked  that  the  number  of  all  eigenpairs  is  given  by  $\frac{(m-1)^{n}-1}{m-2} $ =  
$ M(m,n)$,  which  will  reach  at  the  theoretical  upper  bound  for  the number of   eigenpairs   of  the    symmetric  tensor.   
\end{remark}

So  far,  we   has  answered  the  questions  for   an   orthogonally  diagonalizable symmetric   tensor $\mathcal S  \in    T^{m}(\mathbb R^{n}) $: 
\begin{itemize}
	\item   what  is  each   eigenpair?
	\item    how  much   is  the  number  of  all     eigenpairs? 
\end{itemize}

 Next,  we turn to  analyzing   Theorem  \ref{second_order_necessary},  i.e.,   the  local  optimality  of   each  eigenpair of  the    orthogonally  diagonalizable  symmetric  tensor $\mathcal S   \in    T^{m}(\mathbb R^{n})  $. 
The  first  difficulty   that needs  to  be  solved  is how  to obtain  the  matrix $ \mathbf Q_{2}$.
Since  it is  derived  via  QR factorization, it  can only be  numerically  computed but   not  suitable to be   theoretically  analyzed. 
To deal  with  this  issue,
the  following  lemma is  presented: 
\begin{lemma}\label{phequal}  
	Let  $\mathbf  u	\in \mathbb R^{n\times  1} $   and  $  \mathbf P = 
	\mathbf Q_{2}^{\mathrm T}
	\mathbf H(\mathbf u) 
	\mathbf Q_{2} 
	\in \mathbb R^{(n-1) \times  (n-1)} 
	$  be   the  projected  Hessian  matrix   at  the  point    $\mathbf  u$,  where  $\mathbf H(\mathbf u) $  is  defined  in  (\ref{hessian_matrix}), and    $ \mathbf Q_{2}$ is  calculated  by  (\ref{QR_factor}). 
	Define  a  new  matrix 
	$  \mathbf {M}=  \mathbf  P_{\mathbf  u}^{\bot} \mathbf H (\mathbf u) 
	\mathbf  P_{\mathbf  u}^{\bot}   $,
	where   
	$  \mathbf  P_{\mathbf  u}^{\bot} $ is the orthogonal  cpmplement   projector  of   $\mathbf u$.
	Then,  it     holds  
	\begin{equation}
	\mathbf P   \preceq 0
	\quad 
	\Leftrightarrow
	\quad 
	\mathbf {M}  \preceq 0.	
	\end{equation} 	  
\end{lemma}

\begin{proof}

	First, denote the eigen-decomposition of $ \mathbf P $ as
	\begin{equation}\label{phevd1}
	\mathbf P = 
	\mathbf V
	\mathbf \Lambda 
	\mathbf V^{\mathrm T},
	\end{equation}
	where 
		\begin{equation}\label{lmdeig} 
		\mathbf \Lambda = 
	diag(\sigma_{1},  \sigma_{2},  \dots, \sigma_{\it n-\rm 1}) 
	\end{equation}  is  the  eigenvalue  matrix,
	and 
	$ \mathbf V $ 
	is  an
	$ (n-1) \times  (n-1)$
	orthogonal matrix.
	Furthermore, it can be easily checked that
	\begin{equation}\label{projectionconvert}
	\begin{split}
	\mathbf P_{\mathbf u }^{\bot}
	= 
	\mathbf I- 
	\mathbf u  
	\mathbf u^{\mathrm {T}}
	= 
	\mathbf Q 
	\mathbf Q^{\mathrm {T}} 
	-  
	\mathbf Q_{1} 
	\mathbf R 
	\mathbf R^{\mathrm {T}} 
	\mathbf Q_{1}^{\mathrm {T}} 
	\\   
	=
	\begin{bmatrix}
	\mathbf Q_{1}  &\mathbf Q_{2}
	\end{bmatrix}
	\begin{bmatrix}
	\mathbf Q_{1}^{\mathrm {T}}  \\   \mathbf Q_{2}^{\mathrm {T}}
	\end{bmatrix}
	-
	\mathbf Q_{1}  \mathbf Q_{1}^{\mathrm {T}}     =\mathbf Q_{2}  \mathbf Q_{2}^{\mathrm {T}}
	\end{split},
	\end{equation}
	then it holds
	\begin{equation}\label{eq2}
	\begin{split}
	\mathbf {M}
	&=
	\mathbf P_{\mathbf u }^{\bot} 
	\mathbf H(\mathbf u) 
	\mathbf P_{\mathbf u }^{\bot}
	= 
	\mathbf Q_{2} 
	\mathbf Q_{2}^{\mathrm {T}} 
	\mathbf H(\mathbf u) 
	\mathbf Q_{2}  
	\mathbf Q_{2}^{\mathrm {T}}
	\\
	&=
	\mathbf Q_{2}  
	\mathbf P  
	\mathbf Q_{2}^{\mathrm {T}}    
	=
	\mathbf Q_{2}  
	\mathbf V
	\mathbf \Lambda  
	( \mathbf Q_{2}   \mathbf V)^{\mathrm {T}}  
	=
	\mathbf W   
	\mathbf \Lambda   
	\mathbf W^{\mathrm {T}}
	\end{split} ,
	\end{equation}
	where $   \mathbf W = \mathbf Q_{2}   \mathbf V $ is   an $ n \times  (n-1)$ orthogonal matrix.
	Since   $ \mathbf W $ is singular, whose rank is $ n-1$, we  rewrite  (\ref{eq2}) as
	\begin{equation}\label{eq_evd}
	\begin{split}
	\mathbf {M}
	&=
	\mathbf W   
	\mathbf \Lambda  
	\mathbf W^{\mathrm {T}}    
	=
	\begin{bmatrix}
	\mathbf W  
	&\mathbf v
	\end{bmatrix}
	\begin{bmatrix}
	\mathbf \Lambda  &   0     \\
	0    &   0
	\end{bmatrix}
	\begin{bmatrix}
	\mathbf W^{\mathrm {T}} 
	\\  
	\mathbf v^{\mathrm {T}}
	\end{bmatrix}
	\end{split}
	\end{equation}
	where 
	$ \mathbf v \in  \mathbb R^{n  \times 1} $ 
	is a unit vector, which  lies in the null space of  
	$ \mathbf W^{\mathrm {T}} $, i.e., 
	$ \mathbf W^{\mathrm {T}}
	\mathbf v =0 $,
	and  
	$ \mathbf v^{\mathrm {T}}
	\mathbf v =1 $.
	Then it can proved that 
	$ \mathbf V=  \begin{bmatrix}
	\mathbf W  &\mathbf v
	\end{bmatrix}  $
	is an orthogonal matrix, and thus, the  eigenvalues  matrix  of 
	$ \mathbf P_{\mathbf u }^{\bot} 
	\mathbf H(\mathbf u)  
	\mathbf P_{\mathbf u }^{\bot} $
	are 
	\begin{equation}\label{lmdeigtilde}
\tilde {\mathbf \Lambda } 
=diag(	\sigma_{1},  \sigma_{2},  \dots, \sigma_{n-1}, 0 ) .
	\end{equation}   
	
	In this way, based  on   (\ref{lmdeig})   and  (\ref{lmdeigtilde}), we can conclude that judging    the  positive or  negative  semidefiniteness  of $ \mathbf {P} $ is equivalent to
	judging that   of 
	$ \mathbf {M} $,  and   vice verse.	
\end{proof}

By   using   Lemma  \ref{phequal},  we can     aviod   calculating  the   QR  factorization   of  $\mathbf  u $, and   turn  to   analyzing  the 
positive or  negative  semidefiniteness  of 
$ \mathbf {M} $, which   can  be   explicitly  expressed.         
Now, we   are  interested  in    identifying   that  which   eigenpair    corresponds  to  the  local   extremum  of  
(\ref{opti_ori}). 
In  \cite{hsu},  the  authors  has  considered  the  case  of  $  m=3$.

\begin{theorem}[Theorem 4.2 in  \cite{hsu}]
	\label{hsuresult}
	Let $\mathcal T \in    T^{3}(\mathbb R^{n})   $  have an orthogonal decomposition as given in the form of  
	\begin{equation} \mathcal  T
	=
	\sum_{i=1}^{r}   \lambda_{i} \mathbf  u_{i}^{\circ 3},
	\end{equation}
	and consider the
	optimization problem
	\begin{equation}
	\notag
	\max_{\mathbf u}  \quad \mathcal T \mathbf u^{3}  
	\quad
	\rm s.t.  \quad \mathbf u^{\mathrm {T}}\mathbf u  \le  1. 
	\end{equation}
	Then,  1): the stationary points are eigenvectors of $\mathcal T $,  and  2):  a  stationary point 
	$\mathbf  u $   is an isolated local maximizer if and only if $\mathbf  u =\mathbf  u_{i}  $ ($i=1,2, \dots, r$).
\end{theorem}

The  detailed proof  can   refer  to  \cite{hsu}.
However,  the  authors  only  consider  the  case  of   $ m=3$.   Now, we  are  intersted  in     what  is  the  result    for    the  cases  of   
$ m  \ge  3$.
In  addition,  Theorem  \ref{hsuresult}    only    answers   that  which  eigenvector  is  locally  maximized.
What is the local optimality    of  the   other eigenpairs?     
Here,   we   provide   a  more  generalized   results  for  this  issue. 
And  the  following  theorem  is  presented: 

\begin{theorem} 
	\label{localresult}
	Let $\mathcal S \in    T^{m}(\mathbb R^{n})   $     be     an   orthogonally  diagonalizable symmetric   tensor  as given in the form of   \ref{odt_sum}.   Consider  the  optimization  model   in  (\ref{opti_ori}), for     $\vert    \mathbb  A  \vert  =k ( 1  \le  k  \le  r)    $,  we  have 
	 
		\textbf{case 1:} when  $ k=1$,     $  (\lambda_{i},   \mathbf  u_{i} )  (i=1,2,\dots, r)$       is  the   isolated   local   maximum  solution  of   (\ref{opti_ori}).  
		
		\textbf{case 2:} when  $ 1 <  k  \le  r  < n$,    a  linear   combination of  $  (\lambda_{i},   \mathbf  u_{i} )     $  is  the  saddle  points  of  (\ref{opti_ori}).    
				
		\textbf{case 3:} when  $ k= n$,    a  linear   combination of  $  (\lambda_{i},   \mathbf  u_{i} ) $ ($i=1,2,\dots,n$)      is  the  only   one  isolated    local  mimimum  solution  of  (\ref{opti_ori}). 
	
\end{theorem}

	
	Before  proceeding   our  proof,  we  would  like  to   re-express  (\ref{uUc})  as  the   following  form: 
	\begin{equation}\label{uUc_bot}
	\mathbf u  =    \mathbf U   \mathbf c
	=
	\begin{bmatrix}
	\mathbf U 
	&   \mathbf U_{\bot}
	\end{bmatrix}
	\begin{bmatrix}
	\mathbf c  \\
	\mathbf 0_{n-r}
	\end{bmatrix}
	=
	\tilde {\mathbf U}   
	\tilde {\mathbf c}
	, 
	\end{equation} 
	where 
	$  \mathbf U_{\bot} = [\mathbf u_{r+1}, \mathbf u_{r+2}, \dots, \mathbf u_{n}] 
	 \in  \mathbb R^{n \times (n-r)} $   is  	 a  matrix   that    lies in the null space of  
	$ \mathbf U^{\mathrm {T}}$, i.e., 
	$ \mathbf U^{\mathrm {T}}
	\mathbf u_{j} =\mathbf 0_{r} $,
	and the  column   vectors   of  $  \mathbf U_{\bot} $  satisfy  
	$\mathbf u_{j}^{\mathrm {T}}
	\mathbf u_{j} =1 $, 
		$\mathbf u_{j}^{\mathrm {T}}
	\mathbf u_{k} =0 $ ($j \neq k$) for   $j,k=r+1, r+2, \dots, n$. 
	$\mathbf  0_{n-r}   \in  \mathbb R^{  (n-r)  \times  1  } $    is   a   vector  with  all  elements  are  equal  to  0.  
	$  \tilde {\mathbf U}   = 	
	\begin{bmatrix}
	\mathbf U 
	&   \mathbf U_{\bot}
	\end{bmatrix}
	=   
	[\mathbf u_{1}, \mathbf u_{2}, \dots, \mathbf u_{r}, \mathbf u_{r+1}, \mathbf u_{r+2}, \dots, \mathbf u_{n}]  \in  \mathbb R^{n \times n} $   
	is  an  orthogonal  matrix  that  satisfies 
	$  \tilde {\mathbf U}   \tilde {\mathbf U}^{\mathrm {T}} 
	=   \tilde {\mathbf U}^{\mathrm {T}} \tilde {\mathbf U}  
	=   \mathbf  I_{n}$. 
	$   \tilde {\mathbf c}
	= 	\begin{bmatrix}
	\mathbf c  \\
	\mathbf 0_{n-r}
	\end{bmatrix}
	= 	\begin{bmatrix}
	\mathbf c_{\mathbb A}  \\
	\mathbf 0_{n-k}
	\end{bmatrix}$.

\begin{proof}	
	Based  on   the  notation used  in   (\ref{sum_simple}),  and  by  utilizing  properties (\ref{tensormodemn})  (\ref{tensormodenn}),  we  have   
\begin{equation}
\begin{split}
		\mathcal  S  \mathbf  u^{m-2}
		&=
	( \mathcal D \times_{1} \mathbf U 
	\times_{2} \mathbf U
	\times_{3} 
	\dots 
	\times_{m} \mathbf U)
	\times_{3} (\mathbf  U \mathbf   c)^{\mathrm T} 
	\times_{4} (\mathbf  U \mathbf   c)^{\mathrm T}
	\times_{5} 
	\dots 
	\times_{m} (\mathbf  U \mathbf   c)^{\mathrm T}	
 \\
	&
	=	  \mathcal D
	 \times_{1} \mathbf U 
	\times_{2} \mathbf U 
	\times_{3} ( \mathbf   c^{\mathrm T}\mathbf  U^{\mathrm T} \mathbf  U ) 
	\times_{4} ( \mathbf   c^{\mathrm T}\mathbf  U^{\mathrm T} \mathbf  U ) 
	\dots 
	\times_{m} ( \mathbf   c^{\mathrm T}\mathbf  U^{\mathrm T} \mathbf  U ) 
	\\
	&
	=	  \mathcal D
	\times_{1} \mathbf U 
	\times_{2} \mathbf U 
	\times_{3} ( \mathbf   c^{\mathrm T} ) 
	\times_{4} ( \mathbf   c^{\mathrm T} ) 
	\dots 
	\times_{m} ( \mathbf   c^{\mathrm T} )
		\\
	&
	=	  [\mathcal D
	\times_{3} ( \mathbf   c^{\mathrm T} ) 
	\times_{4} ( \mathbf   c^{\mathrm T} ) 
	\dots 
	\times_{m} ( \mathbf   c^{\mathrm T} ) ]
		\times_{1} \mathbf U 
	\times_{2} \mathbf U 
	\\
	&
		=
		\mathbf  U 
		( \mathcal  D \mathbf   c^{m-2})
		\mathbf  U^{\mathrm T} 
	\\
	&	=
		\mathbf  U  
		\mathbf \Sigma    
		\mathbf  U^{\mathrm T} .
	\end{split}.
	\end{equation}
	where $\mathbf \Sigma$  is an
	$  r  \times  r$  diagonal  matrix, and 
	based on  (\ref{lmdici}),  it holds 
	\begin{equation}\label{a}
	\mathbf \Sigma_{ii} = 
	( \mathcal  D \mathbf   c^{m-2})_{ii} 
	=
	\begin{cases}
	\mathcal  D_{i,i,\dots,i} c_{i}^{m-2}  = 
	\lambda_{i}  c_{i}^{m-2}  =  
	\lambda,  \quad    i   \in  \mathbb  A      \\
	0,    \quad  \quad  \quad   \quad     \quad  \quad   \quad     \quad   \quad       \quad   \quad    \quad     i      \notin  \mathbb  A 
	\end{cases} .
	\end{equation}
	
	Considering  (\ref{uUc_bot}), it  can be  further  rewritten  as 
	\begin{equation}
	\begin{split}
	\mathcal  S  \mathbf  u^{m-2}
	=
	\mathbf  U  
	\mathbf \Sigma    
	\mathbf  U^{\mathrm T}
	= 
		\begin{bmatrix}
	\mathbf U 
	&   \mathbf U_{\bot}
	\end{bmatrix}
	\begin{bmatrix}
	\mathbf \Sigma    &   \mathbf 0  \\
	\mathbf 0   &  \mathbf 0 
	\end{bmatrix}
		\begin{bmatrix}
	\mathbf U 
	&   \mathbf U_{\bot}
	\end{bmatrix} ^{\mathrm T} 
	= 
	 \tilde { \mathbf  U}
	 \begin{bmatrix}
	 \mathbf \Sigma    &   \mathbf 0  \\
	 \mathbf 0   &  \mathbf 0 
	 \end{bmatrix}
	 \tilde { \mathbf  U} ^{\mathrm T}   
	\end{split}.
	\end{equation}
	
	Thus,   we  can  derive     
	\begin{align}\label{m1sum2}
	(m-1)
	\mathcal  S  \mathbf  u^{m-2}
	-  
	\lambda  \mathbf I_{n}
	&  = 
	(m-1)
\tilde { \mathbf  U}
\begin{bmatrix}
\mathbf \Sigma    &   \mathbf 0  \\
\mathbf 0   &  \mathbf 0 
\end{bmatrix}
\tilde { \mathbf  U} ^{\mathrm T} 
	-
	\lambda
	\tilde { \mathbf  U}  
	\tilde { \mathbf  U}^{\mathrm T} 
	\nonumber \\
&	=
	\tilde { \mathbf  U}
[ (m-1)	\begin{bmatrix}
	\mathbf \Sigma    &   \mathbf 0  \\
	\mathbf 0   &  \mathbf 0 
	\end{bmatrix}
	-
	\lambda 
	\mathbf  I_{n} ]
	\tilde { \mathbf  U} ^{\mathrm T} 	
	= 
\tilde { \mathbf  U}
	\hat { \mathbf  \Sigma} 
	\tilde { \mathbf  U} 
  ,
	\end{align}
	where the   elements  of   
	$   \hat { \mathbf  \Sigma}  $   satisfy 
	\begin{equation}\label{eigenclss}
	\hat {\mathbf  \Sigma}_{ii} =   
	 [ (m-1)
		\begin{bmatrix}
	\mathbf \Sigma    &   \mathbf 0  \\
	\mathbf 0   &  \mathbf 0 
	\end{bmatrix}  -\lambda \mathbf   I_{n}  ] _{ii} =
	\begin{cases}
	(m-2)\lambda  ,  \quad  \quad    i   \in  \mathbb A      \\
	-\lambda ,    \quad  \quad  \quad    \quad   \quad     i    \notin   \mathbb A  
	\end{cases}  .
	\end{equation}
Therefore,    it  can  be  further     divided  into  two  blocks,  which  can  be  expressed  as       
\begin{align}\label{m1sum22}
(m-1)
\mathcal  S  \mathbf  u^{m-2}
-  
\lambda  \mathbf I_{n}	
= 
\tilde { \mathbf  U}
\hat { \mathbf  \Sigma} 
\tilde { \mathbf  U} 	
=
\tilde { \mathbf  U} 
\begin{bmatrix}
{ \mathbf  \Sigma}_{\mathbb A} &   \mathbf 0  \\
\mathbf 0     & { \mathbf  \Sigma}_{\mathbb B}
\end{bmatrix}
\tilde { \mathbf  U}^{\mathrm T}   ,
\end{align}
where 
$  { \mathbf  \Sigma}_{\mathbb A}=
(m-2)\lambda  \mathbf  I_{k}  $,  
${ \mathbf  \Sigma}_{\mathbb B}
=-\lambda  \mathbf  I_{n-k}$, 
where  $   \mathbf  I_{k} $ and 
$  \mathbf  I_{n-k} $ is an    $ k \times k $, $ (n-k) \times  (n-k) $  identity  matrix.
	
Next, 	  we consider  
	$   \mathbf  P_{\mathbf  u}^{\bot}  $  in   a   similar  way.

	\begin{align}\label{pubot}
	\mathbf  P_{\mathbf  u}^{\bot}
&	=  \mathbf  I_{n}   -   
	\mathbf  u  \mathbf  u^{\mathrm T}
	=    \mathbf  I_{n}   - 
	 (  \tilde {\mathbf U}  	\tilde {\mathbf c}) 
	 (  \tilde {\mathbf U}   \tilde {\mathbf c}  )^{\mathrm T} 
	 =
	 \tilde {\mathbf U}   \tilde {\mathbf U}^{\mathrm {T}}
	 - (  \tilde {\mathbf U}  	\tilde {\mathbf c}) 
	 (  \tilde {\mathbf U}   \tilde {\mathbf c}  )^{\mathrm T} 
	=
\tilde {\mathbf U}    ( \mathbf  I_{n}  - 	\tilde {\mathbf c} 	\tilde {\mathbf c} ^{\mathrm T}  )  \tilde {\mathbf U}  ^{\mathrm T} 
	 \nonumber  \\
&	=
\tilde {\mathbf U}  
  ( \mathbf  I_{n}  - 	\begin{bmatrix}
 \mathbf c_{\mathbb  A} 
   \\
  \mathbf 0_{n-k}
  \end{bmatrix} 		\begin{bmatrix}
  \mathbf c_{\mathbb  A} 
  \\
  \mathbf 0_{n-k}
  \end{bmatrix} ^{\mathrm T}  )  \tilde {\mathbf U}  ^{\mathrm T} 
	=  \tilde {\mathbf U}
	\begin{bmatrix}
	\mathbf  I_{k}-  \mathbf c_{\mathbb A}  \mathbf   c_{\mathbb A} ^{\mathrm T}    &   \mathbf 0  \\
	\mathbf 0   & \mathbf  I_{n-k}  
	\end{bmatrix}
\tilde {\mathbf U}  ^{\mathrm T}   ,
	\end{align} 
	Then,  it is    derived    
	\begin{equation}\label{PH}
	\begin{split}
	&\mathbf  {M}  =
	\mathbf  P_{\mathbf  u}^{\bot} 
	\mathbf  H(\mathbf x)																			
	\mathbf  P_{\mathbf  u}^{\bot}
	 \\  
	&=
	\tilde {\mathbf U}
	\begin{bmatrix}
	\mathbf  I_{k}-  \mathbf c_{\mathbb A}  \mathbf   c_{\mathbb A} ^{\mathrm T}    &   \mathbf 0  \\
	\mathbf 0   &  \mathbf  I_{n-k} 
	\end{bmatrix}
	\hat { \mathbf  \Sigma}    
	\begin{bmatrix}
	\mathbf  I_{k}-  \mathbf c_{\mathbb A}  \mathbf   c_{\mathbb A} ^{\mathrm T}    &   \mathbf 0  \\
	\mathbf 0   &  \mathbf  I_{n-k} 
	\end{bmatrix}
	\tilde {\mathbf U} ^{\mathrm T}                       \\
	&=  
	\tilde {\mathbf U}
	\begin{bmatrix}
	\mathbf  I_{k}-  \mathbf c_{\mathbb A}  \mathbf   c_{\mathbb A} ^{\mathrm T}    &   \mathbf 0  \\
	\mathbf 0   & \mathbf  I_{n-k}  
	\end{bmatrix}
	\begin{bmatrix}
	{ \mathbf  \Sigma}_{\mathbb A} &   \mathbf 0  \\
	\mathbf 0     & { \mathbf  \Sigma}_{\mathbb B}
	\end{bmatrix}
	\begin{bmatrix}
	\mathbf  I_{k}-  \mathbf c_{\mathbb  A}  \mathbf   c_{\mathbb  A} ^{\mathrm T}      &   \mathbf 0  \\
	\mathbf 0   &    \mathbf  I_{n-k}
	\end{bmatrix}
	\tilde {\mathbf U}  ^{\mathrm T} 
	\\                
	&=
\tilde {\mathbf U}
	\begin{bmatrix}
	(m-2) \lambda  ( \mathbf  I_{k}-   \mathbf c_{\mathbb  A}  \mathbf   c_{\mathbb  A} ^{\mathrm T}  )  &   \mathbf 0  \\
	\mathbf 0     &  { \mathbf  \Sigma}_{\mathbb B} 
	\end{bmatrix}
\tilde {\mathbf U} ^{\mathrm T} 
	\end{split}.
	\end{equation}

	Note  that  
	$    \mathbf  I_{k}-  \mathbf c_{\mathbb  A}  \mathbf   c_{\mathbb  A} ^{\mathrm T} 
	$ 
	is  a  projection  matrix  with  rank  $k-1$, and  it can be  rewritten as  the  form  of   
	$  \mathbf  Q_{\mathbb A}  \mathbf   \Gamma \mathbf  Q_{\mathbb  A}^{\mathrm T}   $,
	where 
	$  \mathbf  Q_{\mathbb  A} $  is    an  $ k \times  k$  orthogonal  matrix,  
	$ \mathbf   \Gamma= diag(1, 1,   \dots,  1,  0)   $   is   an  $ k \times  k$ diagonal  matrix,  where the  number  of  the eigenvalue  of  $ 1 $  is  $ k-1 $.  Then, 
	(\ref {PH})   can be  further  denoted as 
	\begin{equation}\label{phevd}
	\mathbf  {M}  =
\tilde {\mathbf U}
	\begin{bmatrix}
	\mathbf  Q_{\mathbb  A} &   \mathbf 0  \\
	\mathbf 0     &    \mathbf  Q_{\mathbb  B}  
	\end{bmatrix}
	\begin{bmatrix}
	{ \mathbf  \Sigma}_{\mathbb  A}   \mathbf   \Gamma        &   \mathbf 0  \\
	\mathbf 0     &       { \mathbf  \Sigma}_{\mathbb  B} 
	\end{bmatrix}
	\begin{bmatrix}
	\mathbf  Q_{\mathbb  A}^{\mathrm T} &   \mathbf 0  \\
	\mathbf 0     &   \mathbf  Q_{\mathbb  B}^{\mathrm T}  
	\end{bmatrix}
	\tilde {\mathbf U} ^{\mathrm T}   ,
	\end{equation}
	where  $  \mathbf  Q_{\mathbb  B} $  is    an  
	$ (n- k) \times (n- k)$  orthogonal  matrix.
	Thus, (\ref{phevd})  is  the  eigen-decomposition form    of  the    matrix  $ \mathbf  {M} $,  where  
	$     \tilde {\mathbf U}
	\begin{bmatrix}
	\mathbf  Q_{\mathbb  A} &   \mathbf 0  \\
	\mathbf 0     &    \mathbf  Q_{\mathbb  B}  
	\end{bmatrix}  $  
	is  the  eigenvector  matrix,  and  
	$\begin{bmatrix}
	{ \mathbf  \Sigma}_{\mathbb  A}   \mathbf   \Gamma        &   \mathbf 0  \\
	\mathbf 0     &       { \mathbf  \Sigma}_{\mathbb  B} 
	\end{bmatrix}
	$ is  the  eigenvalue  matrix.
	Based  on  (\ref{eigenclss}),    the  $n$  eigenvalues  are
	\begin{equation}\label{eigensign} 
	\underbrace{- \lambda, - \lambda, \dots,  - \lambda}_{n-k} ,
	\underbrace{
		(m-2) \lambda, (m-2) \lambda, \dots,  (m-2) \lambda}_{k-1} , 
	0. 
	\end{equation}
	
	Then,  we  seperately  consider  the sign  of   the  eigenvalues   in   (\ref{eigensign}) according  to  the  value of  $k$:
		
		\textbf{case 1:} when  $ k=1$,   the  eigenpair   falls into  one  of   
		$  (\lambda_{i},   \mathbf  u_{i} )  (i=1,2,\dots, r)$.  Since  $ m \ge 3$,  so    $n$  eigenvalues  are    $  -\lambda$ (the number is  $n-1$)  and  0,   which  are  all  non-positive.   This  implies  that    
		$ \mathbf  {M} \preceq 0 $.  Since  $ \lambda$ is    positive, we conclude   
		$ \mathbf  {P} \prec 0 $.  Thus,    $  (\lambda_{i},   \mathbf  u_{i} )  (i=1,2,\dots, r)$       is  the  isolated   local   maximum  solution  of   (\ref{opti_ori}).

		\textbf{case 2:} when  $ 1 <  k  <  r$,    eigenvalues  are 
		either negative   or  positive,  and  both     $ \mathbf  {M}  $   and  
		$ \mathbf  {P}  $  are  uncertain.  So  a  linear    combination   of  $  (\lambda_{i},   \mathbf  u_{i} )     $  is the  saddle  points  of  (\ref{opti_ori}).  
		
		\textbf{case 3:}
		when   $rank(\mathbf U ) =n$  and  $k=n$,  	
		all    eigenvalues  are  non-negative,  indicating    that     
		$ \mathbf  {M} \succeq 0 $  and   
		$ \mathbf  {P} \succ 0 $.    a  linear   combination of  $  (\lambda_{i},   \mathbf  u_{i} )  ($i=1,2,\dots,n$)    $  is  the  only   one  isolated    local  mimimum  solution  of  (\ref{opti_ori}).  
		
	The  proof  is   completed.
\end{proof}

\section{conclusion}
In this  paper, 
the  Z-eigenpairs  problem of  orthogonally  diagonalizable   symmetric    tensors
is   investigated.  
We  first  show  that   the  eigenpairs  can  be  expressed  in  an  unified  way  and  the  number  of  all  the  eigepairs  is  determined  by  the  order  and  rank  of  the  symmetric  tensor.   
Equipped  with    some    theoretical  analysis, we   eventually    provide  an  unified   proof  for  the  local  optimality of  the  eigenpairs, which  generalizes the  result  in \cite{hsu}.   

\bibliographystyle{elsarticle-num}
\bibliography{ODTREF}

\end{document}